\documentclass[reqno]{amsart}

\usepackage{a4wide, mathrsfs, esint, amssymb, mathabx, enumitem}

\usepackage{mathtools}
\mathtoolsset{showonlyrefs}

\usepackage[colorlinks=true]{hyperref}
\usepackage{color}

\usepackage{graphicx}

\usepackage{amscd}
\usepackage[all]{xy}

\usepackage{tikz-cd}

\usetikzlibrary{decorations.markings}
\tikzset{negated/.style={
		decoration={markings,
			mark= at position 0.5 with {
				\node[transform shape] (tempnode) {$\bigtimes$};
			}
		},
		postaction={decorate}
	}
}

\numberwithin{equation}{section}

\theoremstyle{plain}
\newtheorem{theorem}{Theorem}[section]
\newtheorem{corollary}[theorem]{Corollary}

\newtheorem{lemma}[theorem]{Lemma}

\theoremstyle{definition}
\newtheorem{definition}[theorem]{Definition}
\newtheorem{example}[theorem]{Example}

\theoremstyle{remark}
\newtheorem{remark}[theorem]{Remark}

\newcommand\eps{\ensuremath{\varepsilon}}

\newcommand\g{\ensuremath{\mathbf{g}}}

\newcommand\R{\ensuremath{\mathbb{R}}}

\renewcommand\le{\leqslant}
\renewcommand\ge{\geqslant}

\newcommand\bel[1]{\begin{equation}\label{#1}}
\newcommand\ee{\end{equation}}

\newcommand{\Con}{\mathcal{C}}
\newcommand{\copw}{\mathcal{C}^1_{\mathrm{pw}}}

\newcommand{\vphi}{\varphi}
\newcommand{\cinfty}{\mathcal{C}^\infty}
\newcommand{\sse}{\subseteq}

\begin{document}
\title{The future is not always open}
\author{James D.E.\ Grant}
\email{j.grant@surrey.ac.uk}
\address{Department of Mathematics\\ University of Surrey\\ UK.}
\author{Michael Kunzinger}
\email{michael.kunzinger@univie.ac.at}
\address{Faculty of Mathematics\\ University of Vienna\\ Austria.}
\author{Clemens S\"amann}
\email{clemens.saemann@univie.ac.at}
\address{Faculty of Mathematics\\ University of Vienna\\ Austria.}
\author{Roland Steinbauer}
\email{roland.steinbauer@univie.ac.at}
\address{Faculty of Mathematics\\ University of Vienna\\ Austria.}
\date{\today}

\begin{abstract}
We demonstrate the breakdown of several fundamentals of Lorentzian causality theory in low regularity. 
Most notably, chronological futures (defined naturally using locally Lipschitz curves) may be non-open, and may differ from the corresponding
sets defined via piecewise $\mathcal{C}^1$-curves. By refining the notion of a causal bubble from \cite{CG:12},
we characterize spacetimes for which such phenomena can occur, and also relate these to
the possibility of deforming causal curves of positive length into timelike curves (\emph{push-up}).
The phenomena described here are, in particular, relevant for recent synthetic approaches to
low regularity Lorentzian geometry where, in the absence of a differentiable structure,
causality has to be based on locally Lipschitz curves.
\end{abstract}
\thanks{This work was supported by research grant P28770 of the Austrian Science Fund FWF, and WTZ Project No. CZ 12/2018 of OeAD}
\maketitle
\thispagestyle{empty}

\section{Introduction}

Until a decade ago Lorentzian causality theory had mostly been studied 
under the assumption of a smooth spacetime metric. However, with the analytic 
viewpoint on general relativity becoming more and more prevalent, issues of 
regularity have increasingly come to the fore. In particular, while the geometric 
core of the theory is traditionally formulated in the smooth category, the analysis of the 
initial value problem, for example, requires one to consider metrics of low regularity. 

Beginning in 2011 with the independent works of~\cite{FS:12} and~\cite{CG:12}, a systematic study of Lorentzian causality theory under low 
regularity assumptions began and has brought to light several, sometimes 
surprising, facts. In a nutshell, while some features of causality theory are 
rather robust and topological in nature, other results are usually proved by 
local arguments involving geodesically convex neighbourhoods, which do exist for 
$\mathcal{C}^{1,1}$-metrics~\cite{Min:15,KSS:14}. Consequently for this regularity class 
the bulk of causality theory~\cite{CG:12,Min:15,KSSV:14} including the 
singularity theorems~\cite{KSSV:15,KSV:15,GGKS:18} remains valid. Moreover, 
arguments from causality theory which neither explicitly involve the exponential 
map nor geodesics have been found to extend to locally Lipschitz metrics, see~\cite[Thm.\ 1.25]{CG:12},~\cite[Rem.\ 2.5]{Min:19}. 
However, below $\mathcal{C}^{1,1}$, explicit counterexamples show that convex neighbourhoods may no longer exist 
for metrics of H\"older regularity $\mathcal{C}^{1,\alpha}$ for any $\alpha<1$~\cite{HW:51,SS:18}. 

Moreover, Lorentzian causality theory has been generalized to a theory of 
cone structures, i.e., set-valued maps that assign a cone in the tangent space 
to each point on a manifold. This setting allows one to develop those aspects of 
the theory that merely depend on topological arguments under weak regularity 
assumptions. At the same time, this has led to the introduction of new methods to the 
field: Using weak KAM theory, existence results for smooth time functions 
have been derived for $\mathcal{C}^0$-cone structures in~\cite{FS:12, F:15}, while in~\cite{BS:18} dynamical systems theory has been employed to derive similar 
results for merely upper semi-continuous cone structures. In a landmark paper~\cite{Min:19}, Minguzzi has extended these studies using methods from set-valued 
analysis, convexity and order topology to develop large parts of causality theory for locally Lipschitz cone structures.

Returning to the metric spacetime setting, some results, such as the Avez--Seifert theorem,  
have been extended to $\mathcal{C}^0$-Lorentzian metrics~\cite{Sae:16}. Moreover, a number of fundamental 
works on the $\mathcal{C}^0$-extendibility of spacetimes~\cite{Sbi:18,Sbi:Proc,GLS:18} have appeared recently. However, already in~\cite{CG:12} it was observed that 
two key features of causality may fail for metrics of regularity below Lipschitz:
\begin{enumerate}
\item The push up principle, i.e., $\left( I^+\circ J^+ \right) \cup \left( J^+\circ I^+ \right) \subseteq I^+$, may cease 
to hold\footnote{Here $I^+$ and $J^+$ denote the 
chronological and the causal relation, respectively, and e.g.\ $I^+\circ J^+ =\{(p,r)\in M \times M: 
\exists q\in M$ with $(p,q)\in J^+$ and $(q,r)\in I^+ \}$.};  
\item Light cones may cease to be hypersurfaces as they ``bubble up'' to have 
a nonempty interior.
\end{enumerate}
Moreover, some fundamental questions raised in~\cite{CG:12} remained open, 
especially the openness of the chronological future $I^+(p)$ of a point $p$. In 
this paper, we answer this question (in the negative) and  clarify its
relation to points (1) and (2) above. This leads to a number of natural 
questions concerning the relation between causal and timelike futures and
pasts when defined using various classes of curves (like Lipschitz or piecewise $\mathcal{C}^1$).
To answer these, we construct several counterexamples that exhibit a further
breakdown of standard causality properties beyond those already investigated
in \cite{CG:12}.

Classical accounts on causality theory mainly use piecewise $\mathcal{C}^1$-curves. 
It has to be noted, however, that limit curve theorems, which form a 
fundamental tool in causality theory, fail to respect this regularity class. This has 
led Chru\'sciel in~\cite{Chr:11} to base smooth causality theory entirely on 
the class of locally Lipschitz curves, an approach that has been adopted in most 
of the subsequent studies in low regularity. In fact, Chru\'sciel put it as follows in~\cite[p.\ 14]{Chr:11}:\\

{\it ``In previous treatments of causality theory~\cite{BEE:96,Ger:70,HE:73,ONe:83,Pen:72b,Wal:84} one defines future
directed timelike paths as those paths $\gamma$ which are piecewise differentiable, with $\dot\gamma$ timelike and future 
directed wherever defined; at break points one further assumes that both the left-sided and right-sided derivatives are 
timelike. This definition turns out to be quite inconvenient for several purposes. For instance, when studying
the global causal structure of space-times one needs to take limits of timelike curves,
obtaining thus --- by definition --- causal future directed paths. Such limits will
not be piecewise differentiable most of the time, which leads one to the necessity
of considering paths with poorer differentiability properties. One then faces the
unhandy situation in which timelike and causal paths have completely different
properties. In several theorems separate proofs have then to be given. The approach
we present avoids this, leading --- we believe --- to a considerable simplification of
the conceptual structure of the theory.''}\\

Besides the above, a strong reason for basing causality theory entirely on
locally Lipschitz curves comes from the desire to develop synthetic methods
in Lorentzian geometry. In fact,
following the pioneering works~\cite{Har:82,AB:08}, recently methods from synthetic geometry and, in particular, length spaces have been implemented in the 
Lorentzian setting and in causality theory~\cite{KS:18,GKS:18}. In this framework, which generalizes low regularity 
Lorentzian geometry to the level of metric spaces, $\mathcal{C}^1$-curves are not available. Thus one needs a broad enough framework within which to address the problems highlighted by the examples given
in Section \ref{sec:counter_examples} of the present paper.

This article is structured as follows. In Section~\ref{sec_basic} we define the basic notions of causality theory in several 
variants and establish some fundamental relations between them. We then introduce tools tailored to low regularity, which  enable us to characterise openness of chronological futures and pasts, 
relate non-bubbling with push-up 
and clarify several aspects of pathological behaviour of continuous spacetimes. In Section~\ref{sec:counter_examples} we provide explicit continuous spacetimes that display various pathologies. In particular, we 
give for the first time examples where the chronological future is not open and where the chronological future defined via 
smooth curves and the one defined via Lipschitz curves differ. This answers several open questions raised in~\cite{CG:12}. The final section provides an overview of the interdependence of the
various causality notions studied in this work.

To conclude this introduction we introduce some basic notation. Throughout this paper, $M$ will denote a smooth connected 
second countable Hausdorff manifold. Unless otherwise stated, $g$ will be a continuous Lorentzian metric on $M$, and we will 
assume that $(M,g)$ is time-oriented (i.e., there exists a continuous timelike vector field $\xi$, that is, $g(\xi,\xi)<0$ 
everywhere). We then call $(M,g)$ a~\emph{continuous spacetime}. We shall also fix a smooth complete Riemannian metric $h$ on $M$ 
and denote the induced (length) metric by $d_h$ and the induced norm by $\|.\|_h$. Given Lorentzian metrics $g_1$, $g_2$, we 
say that $g_2$ has \emph{strictly wider light cones\/} than $g_1$, denoted by $g_1\prec g_2$, if for any tangent vector $X 
\neq 0$, $g_1(X,X) \le 0$ implies that $g_2(X,X)<0$ (cf.~\cite[Sec.\ 3.8.2]{MS:08},~\cite[Sec.\ 1.2]{CG:12}). Thus any 
$g_1$-causal vector is timelike for $g_2$.

\section{Curves, bubbles, and open questions}\label{sec_basic}

For Lorentzian metrics of regularity at least $\mathcal{C}^2$, chronological futures and pasts are usually defined via 
piecewise $\mathcal{C}^1$ (or $\mathcal{C}^\infty$)-curves (e.g.,~\cite{HE:73,ONe:83,BEE:96,MS:08}). Especially when working in lower 
regularity, several authors have employed more general classes of curves, the widest one being that of absolutely continuous 
curves, basically due to the fact that this is the largest class of functions for which the Fundamental Theorem of Calculus 
holds. This in turn is required to obtain a reasonable notion of length of curves and to control the curves via their causality.  
To see this let $c\colon[0,1]\rightarrow [0,1]$ be the Cantor function (which has bounded variation) and consider 
the curve $\gamma:$ $[0,1]\ni t\mapsto (c(t),0)$ in two-dimensional Minkowski spacetime $\R^2_1$. Then $\gamma$ parametrizes 
the timelike curve $[0,1]\ni t\mapsto (t,0)$, but its tangent is $(0,0)$ almost everywhere, i.e., spacelike.

Furthermore, absolutely continuous causal curves always possess a locally Lipschitz reparametrisation (cf.\ Lemma~\ref{lem-ac-nic-par} below). On the other hand,
irrespective of the regularity of the metric, the limit curve theorems (\cite{Min:08a}), which are of central importance in 
causality theory, even when applied to families of smooth causal curves typically yield curves that are merely locally 
Lipschitz continuous. 
It is therefore natural to define causal futures and pasts using such curves, and indeed
it was demonstrated in~\cite{Chr:11,CG:12,Min:15,KSSV:14} for spacetimes, as well as in the significantly more general setting of cone structures in~\cite{FS:12,BS:18,Min:19} that a fully satisfactory causality theory can be based on locally Lipschitz causal curves. As we shall see below, 
the situation is more involved in the case of timelike curves. 
 
To discuss the dependence of causality theory on the underlying notions of timelike respectively causal curves, we introduce the following notations:

\begin{definition} We denote by
 \begin{enumerate}
  \item $\mathcal{AC}$ the set of all absolutely continuous curves from an interval into $M$.
  \item $\mathcal{L}$ the set of all locally Lipschitz curves from an interval into $M$.
  \item $\copw$ the set of all piecewise continuously differentiable curves from an interval into $M$.
  \item $\Con^\infty$ the set of all smooth curves from an interval into $M$.
 \end{enumerate}
\end{definition}

Here, a curve is called absolutely continuous if its components in any chart are 
absolutely continuous, or, equivalently, if it is absolutely continuous as
a map into the metric space $(M,d_h)$, where we recall that $d_h$ is the metric on $M$
induced by the Riemannian metric $h$ on $M$ (cf.\ the discussion preceding Thm.\ 6 in~\cite{Min:15}).

Clearly, we have
\begin{equation}\label{eq-cla-cha}
\Con^\infty \subsetneq \copw \subsetneq \mathcal{L} \subsetneq \mathcal{AC}.
\end{equation}
Henceforth, when we say that a curve $\gamma$ in $\mathcal{AC}$ is in fact an element of one of these smaller classes then we mean that there exists a re-parametrization of $\gamma$ that lies in this class.

The definition typically employed in the smooth case (cf.\ the above references),  
but sometimes also used for continuous metrics (e.g.\ in~\cite{Sbi:18}), 
is the following.
\begin{definition}\label{def-cc-cpw}
 Let $\gamma\in\copw$, then $\gamma$ is 
called
\begin{enumerate}
 \item[] \emph{timelike} if $g(\dot\gamma,\dot\gamma)<0$ everywhere.
 \item[] \emph{causal} if $g(\dot\gamma,\dot\gamma)\leq 0$ and $\dot\gamma\neq0$ everywhere.
\end{enumerate}
At breakpoints, the understanding is that the above conditions are satisfied for both the one-sided tangents. A causal curve $\gamma$ is called \emph{future (past) directed} if 
$\dot\gamma$ belongs to the future (past) light cone everywhere (at the breakpoints this means that both one-sided tangents 
belong to the same light cone).
\end{definition}

The usual definition of causal curves of regularity below piecewise $\mathcal{C}^1$ is as follows
(cf.~\cite[Def.\ 1.3]{CG:12}).
\begin{definition}\label{def-cc-ac}
 Let $\gamma\in\mathcal{AC}$, then 
$\gamma$ is 
called
\begin{enumerate}
 \item \emph{timelike} if $g(\dot\gamma,\dot\gamma)<0$ almost everywhere,
 \item \emph{causal} if $g(\dot\gamma,\dot\gamma)\leq 0$ and $\dot\gamma\neq0$ almost everywhere.
\end{enumerate}
A causal curve $\gamma$ is called \emph{future (past) directed} if $\dot\gamma$ belongs to the future (past) light cone 
almost everywhere.
\end{definition}
\begin{remark} In~\cite{BS:18}, there is one further notion of timelike curve that we 
should mention: The authors call a Lipschitz curve timelike if its \emph{Clarke differential\/}
lies in the open chronological cone in the tangent space at each parameter value. 
They then show (\cite[Lem.\ 2.11]{BS:18}) that, using this definition, the chronological
futures and pasts of a point $p\in M$ are precisely the sets $I^\pm_{\Con^\infty}(p)$ (see Definition~\ref{2.5} below).
Therefore, there is no need for a separate treatment of this approach.
\end{remark}

Henceforth we will adhere to the following convention: Whenever a curve is in $\mathcal{AC}$ we will use Definition~\ref{def-cc-ac}, 
however if a curve is \emph{explicitly\/} noted to be in $\copw$, then we use Definition~\ref{def-cc-cpw}.

We now can define the chronological and causal future and past of a point depending on the class of curves chosen.
\begin{definition}
\label{2.5}
Let $\mathcal{A}\in\{\mathcal{AC}, \mathcal{L}, \copw, \Con^\infty\}$ and let $p\in M$. The chronological and causal 
future/past of $p$ with respect to $\mathcal{A}$ are
\begin{align}
I^\pm_{\mathcal{A}}(p)&:=\{q\in M: \exists\text{ future/past directed timelike curve } \gamma\in\mathcal{A} \text{ from 
}p\text{ to }q\}\,,\\
J^\pm_{\mathcal{A}}(p)&:=\{q\in M: \exists\text{ future/past directed causal curve } \gamma\in\mathcal{A} \text{ from 
}p\text{ to }q\}\cup \{p\}\,.
\end{align}
Moreover, for any subset $A\subseteq M$ we set $I^\pm_{\mathcal{A}}(A):=\bigcup_{p\in A} I^\pm_{\mathcal{A}}(p)$ and 
$J^\pm_{\mathcal{A}}(A):=\bigcup_{p\in A}J^\pm_{\mathcal{A}}(p)$.
\end{definition}
From \eqref{eq-cla-cha} we immediately see that for any $p\in M$ we have
\begin{align}
 I^\pm_{\mathcal{C}^\infty}(p)&\subseteq I^\pm_{\copw}(p)\subseteq I^\pm_{\mathcal{L}}(p)\subseteq I^\pm_{\mathcal{AC}}(p)\,,\\
 J^\pm_{\mathcal{C}^\infty}(p)&\subseteq J^\pm_{\copw}(p)\subseteq J^\pm_{\mathcal{L}}(p)\subseteq J^\pm_{\mathcal{AC}}(p)\,.
\end{align}
Next we recall some key notions from~\cite{CG:12} that will be important for our further considerations. 
For any $C>0$, let $\eta_C$ denote the metric $-Cdt^2 + \sum_{i=1}^n (dx^i)^2$ on $\R^{n+1}$ (so $\eta\equiv\eta_1$
is the Minkowski metric).
\begin{definition}\label{cylindric}
A chart $(\varphi=(t,x^1\dots,x^n),U)$ of $M$ centred at $p\in M$ is called cylindrical for $g$ if $U$ is relatively compact, 
$\vphi(U)=L\times V$ for some interval $L\ni 0$ and $0\in V\subseteq \R^n$ open, and
\begin{itemize}
\item[(i)] $(\varphi_*g)(0) = \eta$, the Minkowski metric.
\item[(ii)] There exists some $C>1$ such that $\eta_{C^{-1}}\prec \vphi_* g \prec \eta_C$
on $L\times V$.
\end{itemize}
\end{definition}
Note that point (ii) above implies that $\frac{\partial}{\partial t}$ is timelike and $\frac{\partial}{\partial x^i}$ ($1\le i\le n$) is spacelike on $U$.
By~\cite[Prop.\ 1.10]{CG:12}, every point $p$ lies in the domain of a cylindrical chart. Such a domain is called a 
\emph{cylindrical neighbourhood}.

Following~\cite[Def.\ 1.3]{CG:12}, an absolutely continuous curve $\gamma$ is called~\emph{locally uniformly
timelike\/ (l.u.t.)} if there exists a smooth Lorentzian metric $\check g \prec g$ such that $\check 
g(\dot\gamma,\dot{\gamma})<0$ almost everywhere. For $U\subseteq M$ open and $p\in U$, by $\check I^\pm(p,U)$ we denote the 
set of all points that can be reached by a future (respectively past) directed l.u.t.\ curve in $U$ emanating from $p$. So
\[
\check I^\pm(p,U) = \bigcup \{I^\pm_{\check g}(p,U) : \check g \in \mathcal{C}^\infty, \ \check g \prec g\},
\]
(with $I^\pm_{\check g}(p,U)$ the chronological future/past of $p$ in $U$ with respect to the smooth Lorentzian metric $\check{g}$ 
and using smooth curves). In particular it follows that $\check I^\pm(p,U)$ is open. 

It was shown in the proof of~\cite[Prop.\ 1.10]{CG:12} that in terms of a cylindrical chart $(\vphi,U)$ (where we usually 
will suppress $\vphi$ notationally), $\partial J_{\mathcal{L}}^+(p,U)$ is given as the graph of a Lipschitz function $f_- \colon V \to L$, 
that $I_{\mathcal{L}}^+(p,U)$ is contained in the epigraph $\mathrm{epi}(f_-):=\{(t,x):t\ge f_-(x) \} = J_{\mathcal{L}}^+(p,U)$ of $f_-$,
and that the interior of $J_{\mathcal{L}}^+(p,U)$, denoted by $J_{\mathcal{L}}^+(p,U)^\circ$, equals the strict epigraph $\mathrm{epi}_S(f_-):=\{(t,x):t > f_-(x) \} = J^+(p,U)^\circ$ of $f_-$. Here, $f_-$ is defined as the pointwise limit of the (Lipschitz) graphing
functions of $J^+_{\hat g_k}(p,U)$, where the $\hat g_k\succ g$ are smooth Lorentzian metrics
converging locally uniformly to $g$ as $k\to \infty$. Moreover, the graphing function 
of $\partial \check I^+(p,U)$ is denoted by $f_+$. 
\begin{figure}[h!]
	\begin{center}
		\includegraphics[width=92mm, height= 60mm]{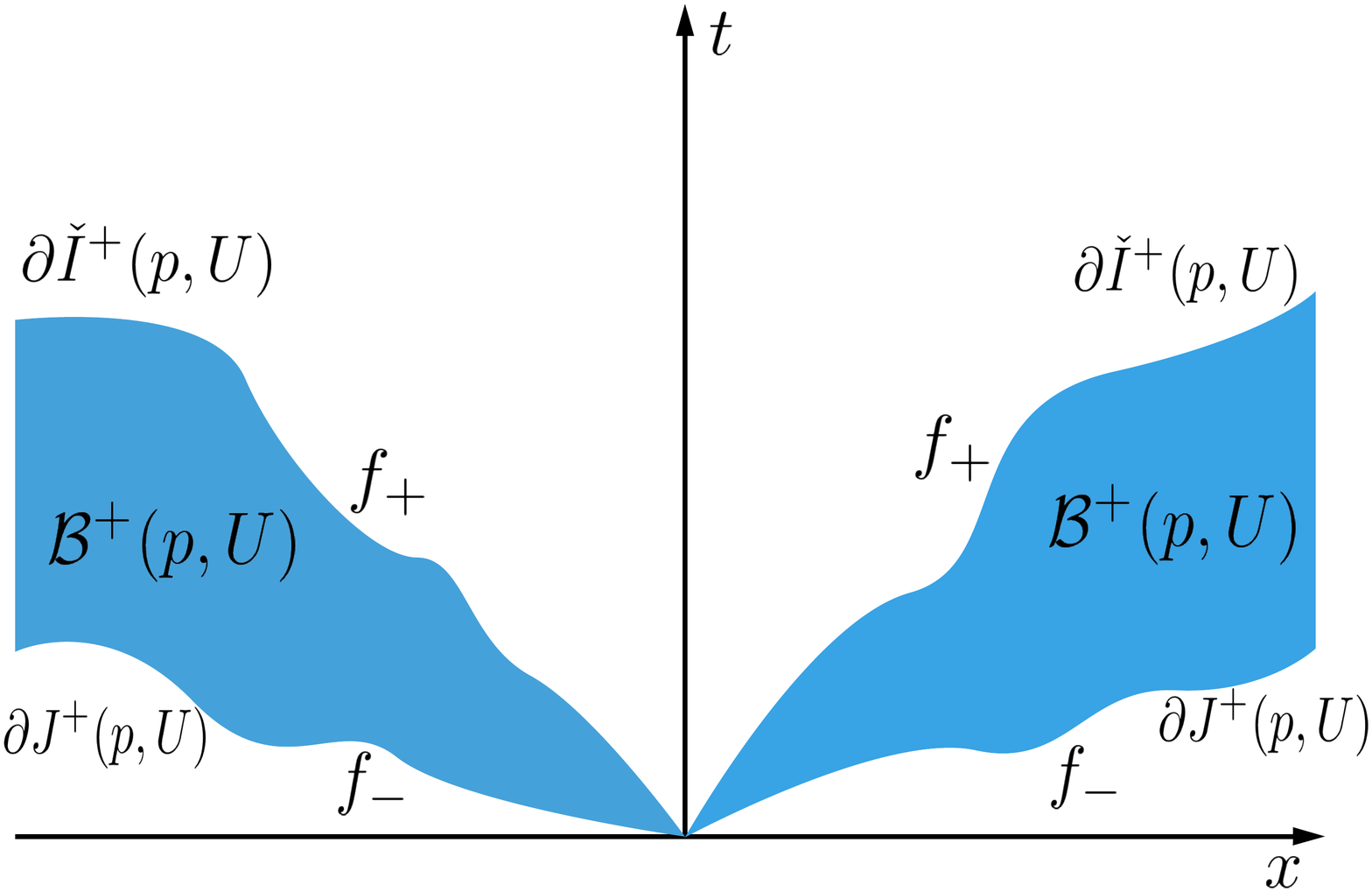}
	\end{center}
	\caption{A cylindrical neighbourhood}
\label{fig:0}
\end{figure}
\begin{lemma}\label{lem:icheck_is_ic1} For any continuous spacetime $(M,g)$ and any $p\in M$,
\begin{equation}\label{icheckic1}
\check I^\pm(p) = I^\pm_{\cinfty}(p) = I^\pm_{\copw}(p).
\end{equation}
In particular, $I^\pm_{\copw}(p)$ is open.
\end{lemma}
\begin{proof}
Clearly, $\check I^\pm(p) \sse I^\pm_{\cinfty}(p) \sse I^\pm_{\copw}(p)$. Conversely, let 
$\gamma \colon [a,b] \to M$ be a future-directed timelike piecewise $\mathcal{C}^1$ curve from $p$ to 
some $q\in I^+_{\copw}(p)$ and set $c:= \max_{t\in [a,b]}g(\dot \gamma(t),\dot \gamma(t))<0$
(with the maximum taken over both values at the finitely many breakpoints). Moreover, we have that $\| \dot\gamma \|_h \leq C'$ on 
$[a,b]$ for some $C'>0$. By~\cite[Prop.\ 1.2]{CG:12} there exists a smooth Lorentzian metric $\check g \prec g$ such that
$d_h(g,\check{g})<\frac{|c|}{2 (C')^2}$, where
\begin{equation}
d_h(g_1,g_2) := \sup_{p\in M,0\not=X,Y\in T_pM} \frac{|g_1(X,Y)-g_2(X,Y)|}{\|X\|_h 	\|Y\|_h}.
\end{equation}
Consequently, $\max_{t\in [a,b]}\check g(\dot \gamma(t),\dot \gamma(t))<c/2<0$ for all $t\in [a,b]$,
so $q\in I_{\check{g}}^+(p)\sse \check I^+(p)$.
\end{proof}

The following Lemma is an analogue of~\cite[Lem.\ 3.2.1]{Kri:99} (providing a distinguished local parametrization for
continuous causal curves in smooth spacetimes) for continuous metrics and cylindrical neighbourhoods:
\begin{lemma}\label{lem-ac-nic-par} 
 Let $(M,g)$ be a continuous spacetime, $p\in M$ and let $(\vphi,U)$ be a cylindrical chart around $p$. Let 
$\gamma\colon[a,b]\rightarrow U$ be an absolutely continuous future directed causal curve. Then there is a reparametrisation 
$\tilde\gamma=\gamma\circ\phi$ of $\gamma$, with $\phi\colon[c,d]\rightarrow[a,b]$ strictly increasing and absolutely 
continuous, such that in the chart $\vphi$, for all $t\in[c,d]$ one has
\begin{equation}
 \tilde\gamma(t) = (t,\vec{\gamma}(t))\,.
\end{equation}
\end{lemma}
\begin{proof}
Suppressing $\vphi$ notationally, let $C>1$ and $g \prec \eta_C$ on $U$. 
Then $\gamma$ is also future directed with respect to $\eta_C$ and hence
 \begin{equation}
  0>\eta_C(\dot\gamma,\partial_t) = -C\dot\gamma^0\,,
 \end{equation}
almost everywhere. Thus $\dot\gamma^0>0$ almost everywhere, and so $\gamma^0$ is strictly monotonically increasing. Setting 
$\phi:=(\gamma^0)^{-1}\colon [\gamma^0(a),\gamma^0(b)]\rightarrow [a,b]$ and observing that $\dot\phi>0$ almost everywhere, 
we obtain that $\phi$ is absolutely continuous by a result of Zarecki~\cite[p.\ 271]{Nat:55}. Consequently, 
$\tilde\gamma:=\gamma\circ\phi$ is absolutely continuous by~\cite[Thm.\ 3, Ch.\ IX, \S 1]{Nat:55} and so $\tilde\gamma$ is a 
future directed causal curve with $\tilde\gamma(t)=(\gamma^0((\gamma^0)^{-1}(t)),\gamma^i(\phi(t))) = (t, \vec\gamma(t))$ 
(note that absolutely continuous functions map sets of measure zero to sets of measure zero --- Lusin's property, see e.g.\ 
\cite[Thm.\ 3.4.3]{AT:04}).
\end{proof}

The previous result allows us to conclude that absolutely continuous causal curves always possess a reparametrisation that
is Lipschitz-continuous. This fact was already noticed in~\cite[Rem.\ 2.3]{Min:19}. We include an alternative proof
based on cylindrical neighbourhoods for convenience.
\begin{lemma}\label{lem:lip_is_ac}
 Let $(M,g)$ be a continuous spacetime. Then any causal curve in $\mathcal{AC}$ is locally Lipschitz continuous, hence lies in 
$\mathcal{L}$. Thus, for any $p\in M$, $I^\pm_{\mathcal{AC}}(p) = I^\pm_{\mathcal{L}}(p)$ and $J^\pm_{\mathcal{AC}}(p) = 
J^\pm_{\mathcal{L}}(p)$.
\end{lemma}
\begin{proof}
Let $\gamma\colon[a,b]\rightarrow M$ be an absolutely continuous future directed causal curve and let $t_0\in[a,b]$. Set 
$p:=\gamma(t_0)$ and let $U$ be a cylindrical neighbourhood of $p$. Moreover, let $\delta>0$ such that 
$\gamma([t_0-\delta,t_0+\delta])\subseteq U$ and let $C>1$ with $g\prec \eta_C$ on $U$. Without loss of generality we can 
assume that $\gamma(t)=(t,\vec\gamma(t))$ on $U$ by Lemma~\ref{lem-ac-nic-par}. Then $\gamma$ is timelike for $\eta_C$ and 
hence
\begin{equation}
 0>\eta_C(\dot\gamma,\dot\gamma) = -C + \|\dot{\vec\gamma}(t)\|_e^2\,,
\end{equation}
where $\|.\|_e$ denotes the Euclidean norm in the chart. This implies that $\|\dot\gamma\|_e\leq \sqrt{1+C}<\infty$. 
Consequently, $\dot\gamma$ is bounded and hence $\gamma$ is 
Lipschitz continuous.
\end{proof}
Based on this result, for $(M,g)$ a continuous spacetime we shall henceforth
define the chronological respectively causal future/past of a point by 
$I^\pm(p):=I^\pm_{\mathcal{AC}}(p)=I^\pm_{\mathcal{L}}(p)$ and $J^\pm(p)
:=J^\pm_{\mathcal{AC}}(p)=J^\pm_{\mathcal{L}}(p)$. This is in accordance with the conventions used in, e.g.,~\cite{Chr:11,CG:12,FS:12,KSSV:14,Sae:16,GGKS:18,GL:18, GLS:18}, and, for the case of $J^\pm(p)$ also with~\cite{BS:18,Min:15,Min:19}. 

For chronological (i.e., possessing no closed timelike curves) spacetimes we have the following characterization of openness of chronological
futures and pasts:
\begin{theorem}\label{th:iplusachronal}
	Let $(M,g)$ be a continuous and chronological spacetime. Then the following are equivalent:
	\begin{itemize}
		\item[(i)] For all $p\in M$, $I^\pm(p)$ is open.
		\item[(ii)] For all $p\in M$, $\partial I^\pm(p)$ is achronal. 
		\item[(iii)] For all $p\in M$, $\partial I^\pm(p)$ is an achronal Lipschitz-hypersurface.
		\item[(iv)] If $F$ is a future/past set with $\partial F\not=\emptyset$, then $\partial F$ 
		is an achronal Lipschitz-hypersurface.
	\end{itemize}
\end{theorem}
\begin{proof} (iii)$\Rightarrow$(ii) is clear.

\noindent (ii)$\Rightarrow$(i): Suppose that $I^+(p)$ were not open for some $p\in M$. Then there exists some $q\in I^+(p)\cap \partial I^+(p)$. As $M$ is chronological, $p\not\in I^+(p)$.
 Since clearly $p\in \overline{I^+(p)}$, it follows that $p\in \partial I^+(p)$. But then the elements $p$ and $q$ of $\partial I^+(p)$ can be connected by a timelike curve, contradicting achronality.

\noindent (i)$\Rightarrow$(iv): Using cylindrical neighbourhoods, this follows exactly as in the smooth case (cf., e.g.,~\cite[Cor.\ 14.27]{ONe:83}), cf.\ also~\cite[Proof of Thm.\ 2.6]{GL:17}. 
For a detailed proof see~\cite[Thm.\ 2.3.5]{Lec:16}.

\noindent (iv)$\Rightarrow$(iii): One only has to note that each $\partial I^+(p)$ is non-empty, 
which follows from $M$ being chronological, cf.\ the argument in (ii)$\Rightarrow$(i) above.
\end{proof}

The spacetime $(M,g)$ is called \emph{causally plain} (\cite[Def.\ 1.16]{CG:12}) if every
$p\in M$ possesses a cylindrical neighbourhood $U$ such that $\partial \check I^\pm(p,U)=\partial J^\pm(p,U)$.
Otherwise it is called \emph{bubbling}\footnote{We note that the phenomenon of bubbling, albeit not under this name, was first observed in \cite{Chr:91}.}. Note that whenever we work in a cylindrical chart, all 
topological notions (like closure, boundary, interior) refer to the relative topology in $U$.
If $\partial \check I^\pm(p,U)\not=\partial J^\pm(p,U)$, then 
the open and non-empty set (\cite[Prop.\ 1.10 (vi)]{CG:12})
\[
\mathcal{B}^+(p,U) := \check I^-(\partial \check I^+(p,U),U) \cap \check I^+(\partial J^+(p,U),U) = \{(t, x) \in U: f_-(x) < t < 
f_+(x)\}
\]
is called the \emph{future bubble set} of $p$ (and analogously for the past bubble set), cf. Figure~\ref{fig:0}. Refining
this terminology, we additionally introduce the \emph{interior future bubble set}
\[
\mathcal{B}^+_{\mathrm{int}}(p,U) := I^+(p,U) \setminus \check I^+(p,U),
\]
as well as the \emph{exterior future bubble set}
\[
\mathcal{B}^+_{\mathrm{ext}}(p,U) := J^+(p,U) \setminus \overline{I^+(p,U)},
\]
and analogously for the past.

Using Lemma~\ref{lem:icheck_is_ic1} and Lemma~\ref{lem:lip_is_ac}, 
it follows that $\mathcal{B}^+_{\mathrm{int}}(p,U)$
consists of those points that can be reached from $p$ by a future directed Lipschitz timelike curve, 
but not by a future directed piecewise $\mathcal{C}^1$ timelike curve. It remained an open problem in~\cite{CG:12} whether spacetimes
with non-trivial interior bubble sets actually exist. We will give examples of this phenomenon in the next section.

The exterior bubbling set is closely connected to push-up properties. To relate these to the concepts introduced above, let
us first give a formal definition:
\begin{definition}\label{def:push-up}
A continuous spacetime $(M,g)$ is said to possess the push-up property if the following holds: Whenever $\gamma \colon [a,b]\to M$
is a (absolutely continuous) future/past directed causal curve from $p=\gamma(a)$ to $q=\gamma(b)$ and if $\{t\in [a,b]: 
\dot\gamma(t) \text{ exists and is future/past directed timelike} \}$ has non-zero Lebesgue measure, then in any neighbourhood of $\gamma([a,b])$ 
there exists a future/past directed timelike Lipschitz curve connecting $p$ and $q$. In particular, $q\in I^+(p)$ 
respectively $q\in I^-(p)$.
\end{definition}

This property implies all the commonly used versions of push-up results. In particular, for an absolutely continuous causal curve 
$\gamma \colon [a,b]\rightarrow M$ the set $\{t \in [a,b]: \dot\gamma(t)$ exists and is future/past directed timelike$\}$ has non-zero 
Lebesgue measure if and only if $\gamma$ has positive length. A spacetime that possesses the push-up property in the sense of Definition~\ref{def:push-up} also satisfies what 
in~\cite[Lemma~2.4.14]{Chr:11} and~\cite[Lemma~1.22]{CG:12} are called
push-up Lemma I (i.e., $I^+(J^+(\Omega)) = I^+(\Omega)$ for any $\Omega\sse M$) and push-up Lemma II, cf.~\cite[Lemma~2.9.10]{Chr:11} and~\cite[Lemma 1.24]{CG:12}. The following result shows that push-up is in fact equivalent to the non-existence
of exterior bubbling:
\begin{theorem}\label{thm:push-up_ext_bubble}
Let $(M,g)$ be a continuous spacetime. The following are equivalent:
\begin{itemize}
\item[(i)] For each $p\in M$ and each cylindrical chart $U$ centred at $p$, $\mathcal{B}^\pm_{\mathrm{ext}}(p,U)=\emptyset$.
\item[(ii)] For each $p\in M$ and each cylindrical chart $U$ centred at $p$, 
$\partial I^\pm(p,U) = \partial J^\pm(p,U)$.
\item[(iii)] $(M,g)$ possesses the push-up property. 
\end{itemize}	
\end{theorem}
\begin{proof} We will use the facts and notations for cylindrical charts introduced prior to Lemma~\ref{lem:icheck_is_ic1}.

\noindent(i)$\Rightarrow$(iii): By covering $\gamma([a,b])$ with cylindrical charts $U$ contained in the given neighbourhood of 
$\gamma([a,b])$ it suffices to show push-up for a curve $\gamma \colon [a,b]\to U$ emanating from $p$ that is future-directed causal 
and such that the set $A:=\{t\in [a,b] : \dot\gamma(t) \text{ exists and is timelike} \}$ has positive Lebesgue-measure 
$\lambda(A)$. Moreover, by Lemma~\ref{lem:lip_is_ac} we only need to work with Lipschitz curves. So we need to establish that 
$q:=\gamma(b)\in I^+(p,U)$. To begin with, suppose that $q\in J^+(p,U)^\circ = \mathrm{epi}_S(f_-)$. Since $I^+(p,U)$ is 
dense in $J^+(p,U)$ by assumption and since $\eta_{C^{-1}}\prec g$ on $U$, we can find an element $q'\in I^+(p,U)$ such that $q$ lies in 
the $\eta_{C^{-1}}$-future of $q'$, and thereby in $I^+(p,U)$ itself. Note that this argument in fact shows that 
$J^+(p,U)^\circ \sse I^+(p,U)$. We are therefore left with the case $q\in \partial J^+(p,U)$. Suppose, then, that such a $q$ were not contained in $I^+(p,U)$. 

We claim that there exists a non-trivial interval $[c,d]\sse (a,b)$ such that $\gamma([c,d])\sse I^+(p,U)$ and that $\bar s\in A$ for some $\bar s\in (c,d)$. To see this, we adapt an 
argument from the proof of~\cite[Prop.\ 1.21]{CG:12}. As was shown there, 
if $t\in A$ and $\gamma(t)\in \partial J^+(p,U)$ then there exists some $\eps>0$ such that $\gamma((t,t+\eps))$ lies in $J^+(p,U)^\circ$ and thereby in $I^+(p,U)$ by the above. 
Arguing analogously to the past, it follows that $\eps$ can be chosen so small that
also $\gamma((t-\eps,t))$ is disjoint from $\partial J^+(p,U)$ (in fact, lies in $\mathrm{hyp}_S(f_-)$). To prove our claim, it suffices
to show that there exists some $s\in A$ with $\gamma(s)\not\in \partial J^+(p,U)$ (hence $\gamma(s)\in
\mathrm{epi}_S(f_-)$), as then
also a non-trivial interval around $s$ is mapped by $\gamma$ to $\mathrm{epi}_S(f_-)$.
So it only remains to exclude the possibility that $\gamma(s)\in \partial J^+(p,U)$ for each
$s\in A$. To do this, we note that since $\lambda(A)>0$, $A$ must contain
accumulation points of $A$ (otherwise $A$ would consist of isolated points and thereby be countable).
However, if $s\in A$ is such an accumulation point and $s_k$ is a sequence in $A$ converging
to $s$, then choosing $\eps$ as above for $s$ we obtain that $s_k\in (s-\eps,s+\eps)\setminus
\{s\}$ for $k$ large, contradicting the assumption that $\gamma(A)\sse \partial J^+(p,U)$.

Now let $f_q^-$ be the graphing function of $\partial J^-(q,U)$.\footnote{It follows from the proof of~\cite[Prop.\ 1.10]{CG:12}
	that the cylindrical neighbourhood $U$ of $p$ can be chosen `locally uniformly' in the sense that each $q\in U$ possesses a
	cylindrical chart with domain $U$.} Then no point $r$ in $\gamma([c,d])$ can lie strictly below the graph of $f_q^-$ 
since otherwise 
$r \in I^-(q,U) \cap I^+(p,U)$, and a fortiori $q\in I^+(p,U)$, contrary to our assumption. Hence $\gamma|_{[c,d]}$ must lie
on the graph of $f_q^-$. But then the same argument as above, using $f_q^-$ and working in the past 
direction leads to a contradiction since at $\gamma(\bar s)$ the curve must enter $I^-(q,U)$. Summing up, we conclude
that $q\in I^+(p,U)$, as claimed.

\noindent(iii)$\Rightarrow$(ii): Since $\partial_t$ is future-directed timelike, (iii) implies that 
$J^+(p,U)^\circ = \mathrm{epi}_S(f_-) \sse I^+(p,U)$, so $\partial I^+(p,U) = \overline{I^+(p,U)}\setminus I^+(p,U)^\circ
\subseteq \partial J^+(p,U)$. Conversely, considering a vertical line through any point $q$ in $\partial J^+(p,U)$ and using
push-up again, it follows
that in any neighbourhood of $q$ there lie points from $I^+(p,U)^\circ$ as well as points from the complement of $\overline{I^+(p,U)}$,
so $q\in \partial I^+(p,U)$.

\noindent(ii)$\Rightarrow$(i): Since $J^+(p,U)=\mathrm{epi}(f_-) = \overline{\mathrm{epi}_S(f_-)}=\overline{J^+(p,U)^\circ}$,
it suffices to show that $J^+(p,U)^\circ \sse I^+(p,U)$. Suppose, to the contrary, that there exists some $q\in 
J^+(p,U)^\circ \setminus I^+(p,U)$. We can then connect $q$ via a continuous path $\alpha$ that runs entirely in $J^+(p,U)^\circ$
with a point in $I^+(p,U)$ (e.g., with any point on the positive $t$-axis). But then $\alpha$ has to intersect
$\partial I^+(p,U)$, producing a point in $\partial I^+(p,U)\setminus \partial J^+(p,U)$, a contradiction.
\end{proof}
\begin{remark} We note that the equivalence between the push-up property and the absence
of causal bubbles was first observed, even in the more general setting of closed cone
structures (with chronological futures/pasts defined via piecewise $\Con^1$-curves) 
by E.\ Minguzzi in \cite[Thm.\ 2.8]{Min:19}. 
\end{remark}

Concerning the relationship between the various bubble sets we have:
\begin{lemma}\label{lem:bubble_inclusions}
Let $(M,g)$ be a continuous spacetime and let $U$ be a cylindrical neighbourhood of $p\in M$. 
Then 
\begin{equation}\label{eq:bubble_inclusion}
\{ f_-(x) < t <  f_+(x) \} = \mathcal{B}^+(p,U) \sse
\overline{\mathcal{B}^+_{\mathrm{int}}(p,U)} \cup \mathcal{B}^+_{\mathrm{ext}}(p,U) \sse \{ f_-(x) \le t \le f_+(x)\},
\end{equation}
and analogously for the past-directed case.
\end{lemma}
\begin{proof} By~\cite[Prop.\ 1.10]{CG:12} and~\cite[Prop.\ 1.21]{CG:12}, the bubble set $\mathcal{B}^+(p,U)$ is the 
intersection of $\mathrm{epi}_S(f_-)$ and the strict hypograph $\mathrm{hyp}_S(f_+)$ of the graphing function 
$f_+$ of $\partial \check I^+(p,U)$. From this, the second inclusion follows immediately. 
To see the first one, note that by the properties of cylindrical neighbourhoods detailed
before Lemma \ref{lem:icheck_is_ic1}, any $p\in \mathcal{B}^+(p,U)$ is contained
in $J^+(p,U)\setminus \overline{\check I^+(p,U)}$. Now suppose, in addition, that $p\not\in \overline{\mathcal{B}^+_{\mathrm{int}}(p,U)} \supseteq \overline{I^+(p,U)}\setminus
\overline{\check I^+(p,U)}$. Then $p\in J^+(p,U)\setminus \overline{I^+(p,U)} = \mathcal{B}^+_{\mathrm{ext}}(p,U)$.
\end{proof}

From Lemma~\ref{lem:icheck_is_ic1} and Lemma~\ref{lem:lip_is_ac} we obtain:
\begin{theorem}\label{thm:all_i_equal_for_causally_plain}
Let $(M,g)$ be a continuous spacetime. The following are equivalent:
\begin{itemize}
\item[(i)] For any $p\in M$, 
\[
I^\pm_{\mathcal{C}^\infty}(p) = I^\pm_{\copw}(p) = I^\pm_{\mathcal{L}}(p) =  I^\pm_{\mathcal{AC}}(p).
\]
\item[(ii)] There is no internal bubbling, i.e., $\mathcal{B}^{\pm}_{\mathrm{int}}(p,U)=\emptyset$ for every $p\in M$ and every 
cylindrical neighbourhood $U$ of $p$.
\item[(iii)] $I^\pm(p)=\check{I}^\pm(p)$ for every $p\in M$.
\end{itemize}
\end{theorem}

\begin{corollary}
\label{cor:2.15}
Let $(M, g)$ be a continuous spacetime. The following are equivalent: 
\begin{enumerate}[label={(\roman*)}]
\item $(M, g)$ is causally plain. 
\item There is neither internal nor external bubbling, i.e. $\mathcal{B}^{\pm}_{\mathrm{int}}(p,U) = \mathcal{B}^{\pm}_{\mathrm{ext}}(p,U)=\emptyset$, for all $p \in M$. 
\end{enumerate}
\end{corollary}
\begin{proof}
Let $(M, g)$ be causally plain. By~\cite[Prop.~1.21]{CG:12}, we have $\check{I}^{\pm}(p) = I^{\pm}(p)$ for all $p \in M$, so there is no interior bubbling. By~\cite[Lemma~1.22]{CG:12}, $(M, g)$ has the push-up property, so, by Theorem~\ref{thm:push-up_ext_bubble}, we have $\mathcal{B}^{\pm}_{\mathrm{ext}}(p,U) = \emptyset$. 

The other direction follows directly from Lemma~\ref{lem:bubble_inclusions}. 
\end{proof}

Any spacetime with a Lipschitz continuous metric is causally plain by~\cite[Cor.\ 1.17]{CG:12} and so, by 
Corollary~\ref{cor:2.15}, such a spacetime has neither internal nor external bubbling. Moreover, by 
Theorem~\ref{thm:all_i_equal_for_causally_plain}, for spacetimes without internal bubbling, it does not matter which type of curves is used in the definition of chronological futures and pasts.

After these preparations we can now formulate the questions that will be addressed in the remainder of this paper.  
Let $(M,g)$ be a continuous spacetime. The following two questions have already 
been posed in~\cite{CG:12}:
\begin{enumerate}[label={(Q\arabic*)}, ref={Q\arabic*}]
\item\label{Q1}
Is $I^\pm(p)$ open for each $p$?
\item\label{Q2}
Is $I^\pm(p)=\check I^\pm(p)$ for each $p$?
\end{enumerate}
According to Theorem~\ref{thm:all_i_equal_for_causally_plain},~(\ref{Q2}) is equivalent to the question
\begin{itemize}
\item[(Q2${}^{\prime}$)] Is it true that, for all $p\in M$, $I^\pm_{\mathcal{C}^\infty}(p) = I^\pm_{\copw}(p) = 
I^\pm_{\mathcal{L}}(p) =  I^\pm_{\mathcal{AC}}(p)$?
\end{itemize}
Obviously, for spacetimes such that~(\ref{Q2}) can be answered affirmatively, the same is true for~(\ref{Q1}). It is
natural to ask for the converse of this implication:
\begin{enumerate}[resume*]
\item\label{Q3}
Does an affirmative answer to~(\ref{Q1}) for a given spacetime imply the same for~(\ref{Q2})?
\end{enumerate}
Heuristically speaking, to call a set a bubble, one would expect it to have a non-empty interior. A natural question that therefore arises is the following. 
\begin{enumerate}[resume*]
\item\label{Q4}
Is $\check{I}^+(p)$ dense in $I^+(p)$, i.e., is $I^\pm(p) \subseteq \overline{\check I^\pm(p)}$, for all 
$p \in M$? 
\end{enumerate}
Finally, we address the relation between bubbling and interior bubbling: As noted above, non-bubbling,
i.e., causally plain, spacetimes cannot exhibit interior bubbling. On the other hand, in the only 
currently known examples of bubbling metrics (\cite[Ex.\ 1.11]{CG:12}) there is no interior bubbling,
so all bubble sets are exterior (i.e., lie outside of $\overline{I^+(p,U)}$). Another natural question is therefore:
\begin{enumerate}[resume*]
\item\label{Q5}
Given a non-causally plain spacetime, are all bubble sets of the same type (interior or exterior)? 
\end{enumerate}
In the following section, we provide examples that answer all of these questions in the negative.

\section{Counterexamples}\label{sec:counter_examples}

In this section, we present examples of Lorentzian metrics which show that the answers to the questions~(\ref{Q1})--(\ref{Q5}) are negative. The metrics are all defined on a two-dimensional domain $\Omega \subseteq \R^2$ with coordinates $(t, x)$, and are taken to be of the form 
\bel{metric}
\g := 2 \left[ - \sin 2\theta(t, x) \, dt^2 - 2 \cos 2\theta(t, x) \, dx \, dt + \sin 2\theta(t, x) \, dx^2 \right]
\ee
for an appropriately chosen function $\theta(t, x)$. Note that this metric has the property that the vectors 
\[
v_1 := \cos\theta(t, x) \, \partial_x + \sin\theta(t, x) \, \partial_t, 
\qquad 
v_2 := \cos\theta(t, x) \, \partial_t - \sin\theta(t, x) \, \partial_x
\]
are null, and $\g(v_1, v_2) = -2$. 

\begin{example}
\label{number1}
Let $\Omega := \R^2$ and $0 < \alpha < 1$. We define the function 
\[
\theta(t, x) \equiv \theta(x) := \begin{cases} 0, &x < -1,\\ 
\arccos |x|^{\alpha}, &-1 \le x \le 0,\\ 
\frac{\pi}{2}, & x > 0.\end{cases}
\]
With this function $\theta$, the metric~\eqref{metric} is $\alpha$-H\"{o}lder continuous, but not Lipschitz. It is smooth away from $x=0$ and $x=-1$. If we wish to confine the non-smooth behaviour of the metric to the $t$-axis alone, we may restrict to the region $x \in (-1, \infty)$. Alternatively, we may simply smooth out the metric near the set $x=-1$.  

The light cones of this metric are illustrated in Figure~\ref{fig:1}. 
\begin{figure}[h!]
	\begin{center}
		\includegraphics[width=92mm, height= 60mm]{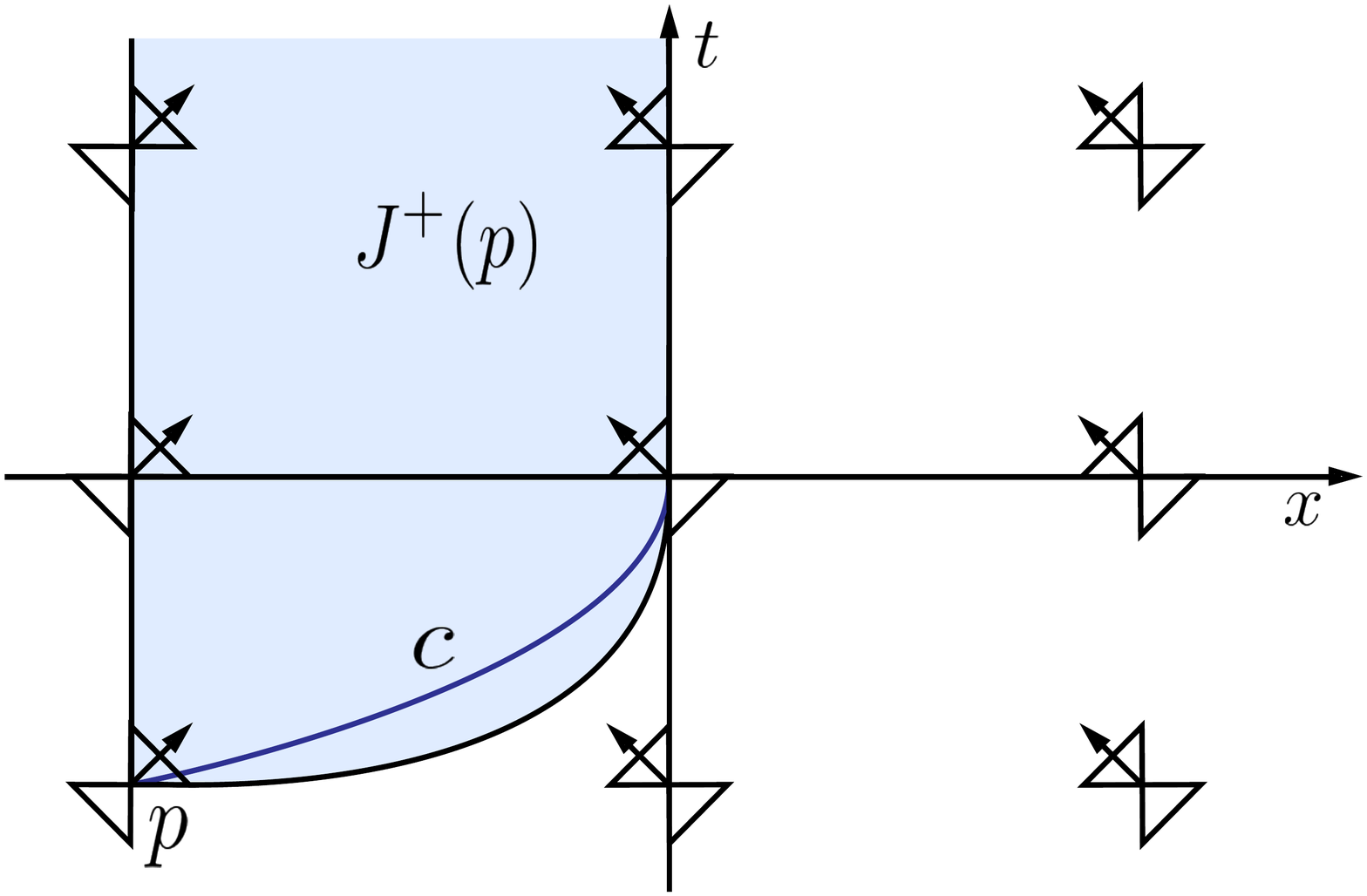}
	\end{center}
	\caption{Timelike curve that reaches $\partial J^+(p)$}
\label{fig:1}
\end{figure}
If we consider a null generator parametrised by its $x$-value leaving a point $p$ with $t(p) = t_0$, $x(p) = -1$ tangent to the $x$-axis, then its tangent vector will be proportional to $v_1$ and, as such, $t(x)$ will satisfy the ordinary differential equation
\[
\frac{dt}{dx} = \tan \theta(t, x) = \frac{\left( 1 - |x|^{2 \alpha} \right)^{1/2}}{|x|^{\alpha}}, 
\]
for $x \in (-1, 0]$ with $t(-1) = t_0$. Therefore, the generator will reach the $t$-axis, i.e. $x = 0$, at time 
\[
t_1 := t_0 + \int_{-1}^0 \frac{\left( 1 - |x|^{2 \alpha} \right)^{1/2}}{|x|^{\alpha}} \, dx < \infty. 
\]
 Note that this argument crucially relies on the fact that $\alpha < 1$. If $g$ were Lipschitz ($\alpha = 1$), for example, the above curve does not reach the $t$-axis in finite time. This is consistent with the fact that Lipschitz metrics are causally plain (cf.~Corollary~\ref{cor:2.15}). 
 
 The generator constructed above reaches the $t$-axis at a finite value $t_1>t_0$, and can then be continued vertically up the $t$-axis as a null generator of $\partial J^+(p)$. Indeed, we note that any absolutely continuous, future-directed, causal curve from a point on the $t$-axis cannot enter the region $x>0$. In particular, let $\alpha \colon [0, \infty) \to \R^2$ be an absolutely continuous, future-directed, causal curve with $\left( x \circ \alpha \right)(0) = 0$, and assume that there exists $s_0 > 0$ such that $\left( x \circ \alpha \right)(s_0) > 0$. Then, there exists a subset $B \subseteq [0, \infty)$ with positive Lebesgue measure such that for all $s \in B$, $\frac{d}{ds} \left( x \circ \alpha \right)(s)$ exists and is strictly positive. (If this were not the case, then $\left( x \circ \alpha \right)(s_0) = 0 + \int_0^{s_0} \frac{d}{ds} \left( x \circ \alpha \right)(s) \, ds \le 0$.) However, any vector of the form $\left( \frac{dt}{ds}, \frac{dx}{ds} \right)$ with $\frac{dt}{ds} \ge 0$ and $\frac{dx}{ds} > 0$ is not future-directed causal in the region $x \ge 0$. This contradicts the assumption that the curve $\alpha$ is future-directed causal, i.e. that $\dot{\alpha}(
s)$ is future-directed causal almost everywhere.

The preceding observation implies that the set $J^+(p)$ consists of the light blue region in Figure~\ref{fig:1}, along with the vertical null generator from $p$ and the right-moving null generator from $p$. (In particular, the subset of the $t$-axis with $t \ge t_1$ is part of the boundary of $J^+(p)$.) 

We now note, however, that any point $(t, 0)$ in the null cone of $p$ with $t > t_1$ can also be reached by a curve from $p$ that is $\mathcal{C}^1$ and timelike at all points except the intersection of the curve with the $t$-axis. In particular, let $q = (\overline{t}, 0)$ with $\overline{t} > t_1$. Then $q \in \partial I^+(p)$. Consider the curve $\gamma \colon (-\eps, 0] \to M$, $s \mapsto \left( t(s), x(s) \right)$ with 
\[
t(s) = \overline{t} + \frac{1}{1-\alpha} A^{1-\alpha} s, \qquad x(s) = - A |s|^{\frac{1}{1-\alpha}}, 
\]
where $A > 0$ is constant. Note that $x(s) < 0$ for all $s \in (-\eps, 0)$ and $x(0) = 0$. Moreover, $\left. \frac{dt}{ds} \right|_{s=0} \neq 0$, $\left. \frac{dx}{ds} \right|_{s=0} = 0$ and, since $\left. \g \right|_{x=0} = 4 \, dt \, dx$, it follows that $\frac{d\gamma}{ds}$ is null at this point. Finally, noting that 
\[
\cos\theta(x(s)) = |x(s)|^{\alpha} = A^{\alpha} \, |s|^{\frac{\alpha}{1-\alpha}}, \qquad \sin\theta(x(s)) = \left( 1-A^{2\alpha} \, |s|^{\frac{2\alpha}{1-\alpha}} \right)^{1/2}, 
\]
we now calculate that 
\begin{align*}
\g \left( \frac{d\gamma}{ds}, \frac{d\gamma}{ds} \right) 
&= -2 \left[ 2 A^{\alpha} \, |s|^{\frac{\alpha}{1-\alpha}} \left( 1-A^{2\alpha} \, |s|^{\frac{2\alpha}{1-\alpha}} \right)^{1/2} \left( - \frac{A^2}{(1-\alpha)^2} |s|^{\frac{2\alpha}{1-\alpha}} \right) \right. 
\\
&\hskip 3cm \left.  
+ 2 \left( A^{2\alpha} \, |s|^{\frac{2\alpha}{1-\alpha}} - 1 \right)  \, \frac{A^{2-\alpha}}{(1-\alpha)^2} |s|^{\frac{\alpha}{1-\alpha}} \right] 
\\
&= - \frac{4}{(1-\alpha)^2} A^{2-\alpha} |s|^{\frac{\alpha}{1-\alpha}} \left[ 2 A^{2\alpha} \, |s|^{\frac{2\alpha}{1-\alpha}} - 1 + \left( 1 - A^{2\alpha} \, |s|^{\frac{2\alpha}{1-\alpha}} \right)^{3/2} \right]
\\
&= - \frac{4}{(1-\alpha)^2} A^{2-\alpha} |s|^{\frac{\alpha}{1-\alpha}} \left[ \frac{1}{2} A^{2\alpha} \, |s|^{\frac{2\alpha}{1-\alpha}} + O(|s|^{\frac{4\alpha}{1-\alpha}}) \right]. 
\end{align*}
As such, the curve $\gamma$ is timelike for small $s < 0$ and null only at $s=0$. It follows that, for small $s_0 < 0$, the point $\gamma(s_0)$ lies in $I^+(p)$. Since the metric $g$ is smooth for $x < 0$, there exists a $\mathcal{C}^1$ timelike curve from $p$ to $\gamma(s_0)$. We may assume, without loss of generality, that the tangent vector of this curve coincides with that of the curve $\gamma$ at the point $\gamma(s_0)$. Concatenating this curve with the restriction of $\gamma$ to $(s_0, 0]$ gives a curve $c \colon [-1, 0] \to \R^2$ with $c(-1) = p$, $c(0) = q \in \partial J^+(p)$, that is $\mathcal{C}^1$ with timelike tangent vector on the open interval $(-1, 0)$ and whose tangent vector at $q$ is null. 

This example has the following properties: 
\begin{itemize}[itemsep=2mm]
\item The chronological future of any point with $x < 0$ is not an open subset of $\R^2$. In particular, $I^+(p)$ consists of the light blue region in Figure~\ref{fig:1}, \emph{including the subset of the $t$-axis with $t > t_1$}. This answers~(\ref{Q1}) in the negative. 
\item $\check{I}^+(p)$ consists of the light blue region in Figure~\ref{fig:1}, \emph{excluding\/} the subset of the $t$-axis with $t > t_1$. As such, $\check{I}^+(p) \neq I^+(p)$. This answers~(\ref{Q2}) in the negative. 
\item The exceptional curve $c$ constructed in the example may be taken to be smooth, with only one point at which it is null. It may be extended in a timelike fashion to give a curve that is smooth except at one point of non-differentiability, and timelike everywhere where its derivative is defined (at which point, the tangent vector from the past direction is null and the tangent vector to the future is timelike). As such, the curve is not timelike in the sense of Definition~\ref{def-cc-cpw}, but it~\emph{is\/} timelike in the sense of Definition~\ref{def-cc-ac}. 
\item Since $\check{I}^+(p)$ is not equal to $I^+(p)$, it follows from Theorem~\ref{thm:all_i_equal_for_causally_plain} and Corollary~\ref{cor:2.15} that the spacetime cannot be causally plain, so there must exist a point $r \in M$ with $\mathcal{B}^{\pm}(r, U) \neq \emptyset$, for some cylindrical neighbourhood $U$ of $r$. Indeed, while the bubble set of $p$ is empty, every point on the $t$-axis has a non-empty past bubble set. This behaviour will be even more pronounced in the following example. 
\end{itemize}
\end{example}

\begin{example}
\label{number2}
Let $(t, x) \in \Omega = \R^2$, and $0 < \alpha < 1$. We take the metric to be of the form~\eqref{metric}, but with  
\[
\theta(t, x) \equiv \theta(x) := \begin{cases} 0, &x < -1, \\ 
\arccos |x|^{\alpha}, &-1 \le x \le 1, \\ 
0, & x > 1. \end{cases}
\]
As in Example~\ref{number1}, the metric~\eqref{metric} is $\alpha$-H\"{o}lder but not Lipschitz, and smooth away from $x = 0, \pm 1$. 

The light cone structure for this metric is illustrated in Figure~\ref{fig:2}. 
\begin{figure}[h!]
	\begin{center}
		\includegraphics[width=92mm, height= 60mm]{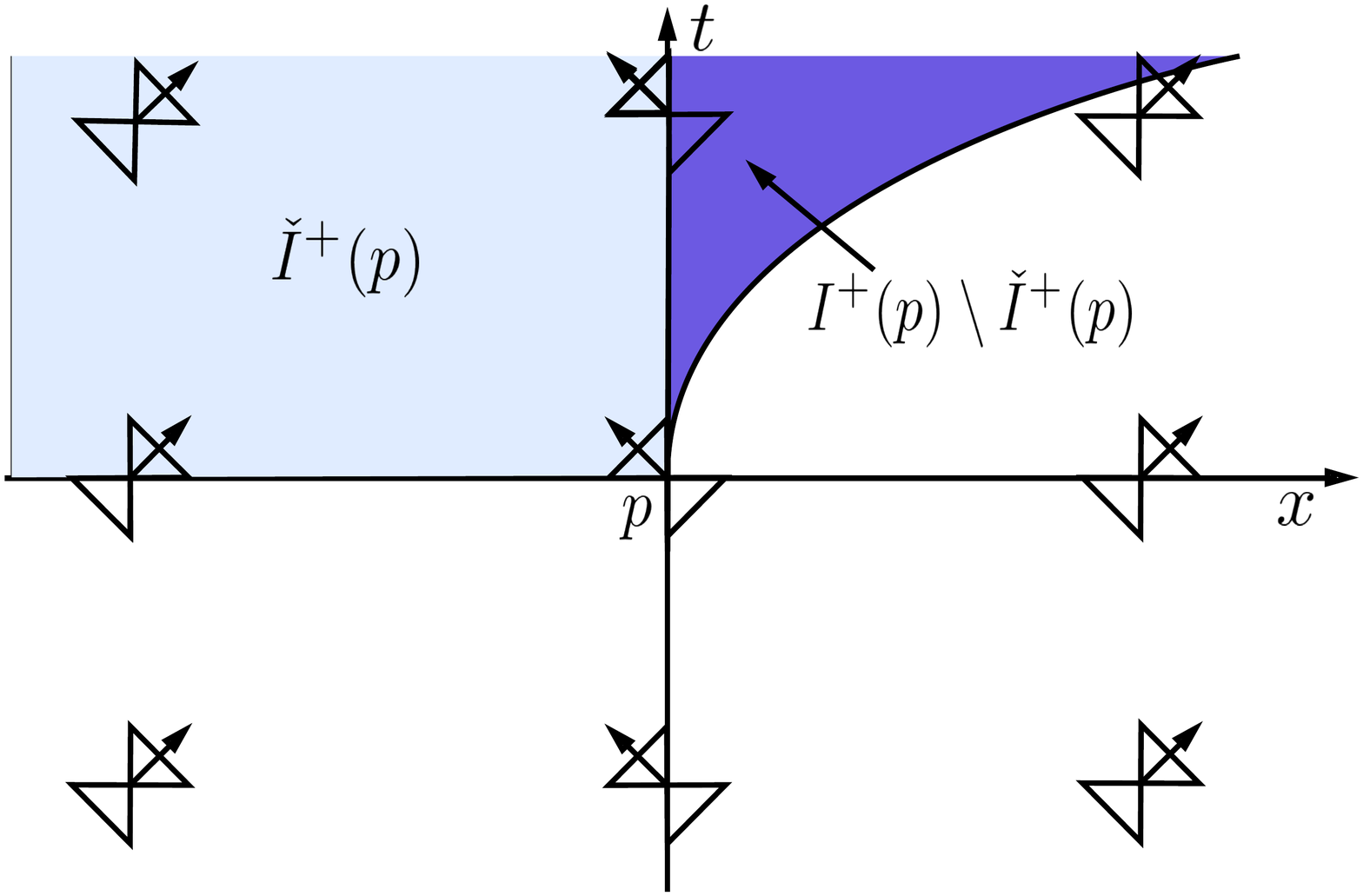}
	\end{center}
	\caption{Interior bubbling}
\label{fig:2}
\end{figure}
For the point $p$ in the figure, the light blue set is the set $\check{I}^+(p)$, while the union of the light blue and dark blue sets (including all points on the $t$-axis above the point $p$) is the set $I^+(p)$. As such, non-trivial internal bubbling occurs for this metric. In particular, 
\begin{itemize}[itemsep=2mm]
\item The set $I^+(q)$ is open for all points $q \in \R^2$. However, for the particular point $p$ in Figure~\ref{fig:2}, we have that $I^\pm(p) \neq \check I^\pm(p)$. This shows that the answer to~(\ref{Q3}) is negative. 
\item For the point $p$ in Figure~\ref{fig:2}, the set $\check{I}^+(p)$ is contained in the region $x < 0$, while $I^+(p)$ contains a subset of the region $x>0$ with non-empty interior. It follows that $\check{I}^+(p)$ is not a dense subset of $I^+(p)$ in this spacetime. Therefore, this example answers~(\ref{Q4}) in the negative. 
\end{itemize}
\end{example}

\begin{example}
\label{number3}
In~\cite{CG:12}, the $\mathcal{C}^{0, \lambda}$ metric 
\bel{CGmetric}
\g = - \left( du + \left( 1 - |u|^{\lambda} \right) dx \right)^2 + dx^2
\ee
for $(u, x) \in \R^2$ with $0 < \lambda < 1$ was considered, which exhibits external bubbling. In our next example, we define a different metric where the external bubbling effect is localised. We then glue this metric into Example~\ref{number2} above to construct a metric where there exists a point $p \in M$ such that $\check{I}^+(p) \subsetneq I^+(p)$ and $\overline{I^+(p)} \subsetneq J^+(p)$ at the same time. 

Let $x \in \R$, $t \in \R$, and $0 < \alpha, \lambda < 1$. Let $\theta_0 := \arccos \left( \frac{1}{2} \right)^{\alpha}$, and fix $\rho \in \left( 0, 1 \right)$ such that $\arctan \rho^{\lambda} < \theta_0$. We define the regions 
\[
A := \left\{ (t, x) \in \left[ -\rho, \rho \right] \times \left[ -\frac{5}{6}, -\frac{2}{3} \right] \right\}, \qquad 
B := \left\{ (t, x) \in \left[ -1, 1 \right] \times \left[ -1, -\frac{1}{2} \right] \right\}. 
\]

We consider the Lorentzian metric~\eqref{metric} where the function $\theta(t, x)$ is defined as follows: 
\begin{itemize}
\item[$\bullet$] For $(t, x) \not\in B$, we define 
\[
\theta(t, x) := \begin{cases} 0, &x < -1, \\ 
\arccos |x|^{\alpha}, &-1 \le x \le 1, \\ 
0, & x > 1; \end{cases}
\]
\item[$\bullet$] For $(t, x) \in A$, let 
\[
\theta(t, x) := \arctan |t|^{\lambda}; 
\]
\item[$\bullet$] We choose the function $\theta$ on the set $B \setminus A$ in such a way that 
\begin{enumerate}[label={(\roman*). }, itemsep=2mm]
\medskip
\item $\theta$ is continuous on $M \equiv \R^2$; 
\item For each fixed $t \in [-1, 1]$ the angle $\theta(t, x)$ is a non-decreasing function of $x \in [-1, 0]$, i.e., for each fixed $t \in \left[ - \rho, \rho \right]$, the function $x \mapsto \theta(t, x)$ is non-decreasing; 
\item The light cones match up to those on the boundary of the sets $A$ and $B$. In particular, we require
\[
\theta\left( t, -1 \right) = 0, \qquad \theta\left( t, -\frac{1}{2} \right) = \arccos \left( \frac{1}{2} \right)^{\alpha}, \qquad \forall t \in [-1, 1], 
\]
and
\[
\theta\left( t, -\frac{5}{6} \right) = \theta\left( t, -\frac{2}{3} \right) = \arctan |t|^{\lambda}, \qquad \forall t \in [-\rho, \rho]. 
\]
\end{enumerate}
\end{itemize}

\begin{remark}
The function $\theta(t, x)$ may be chosen to be smooth away from the $t$-axis and the points of the set $A$ that lie on the $x$-axis. Globally, we then have that $\theta$ has H\"older regularity $\mathcal{C}^{0, \beta}$, with $\beta := \min (\alpha, \lambda)$. 
\end{remark}

The right-moving null generators leaving a point $t=0$, $x \in (-\frac{5}{6}, -\frac{2}{3})$ in the region $A$ satisfy 
\[
\frac{dt}{dx} = |t|^{\lambda}. 
\]
The non-uniqueness of the solutions to this equation leads to the same external bubbling effect demonstrated by the metric~\eqref{CGmetric} (cf.~\cite{CG:12}). In our case, however, this non-uniqueness effect is localised within the region $A$. 
\end{example}

\begin{figure}[h!]
	\begin{center}
		\includegraphics[width=92mm, height= 60mm]{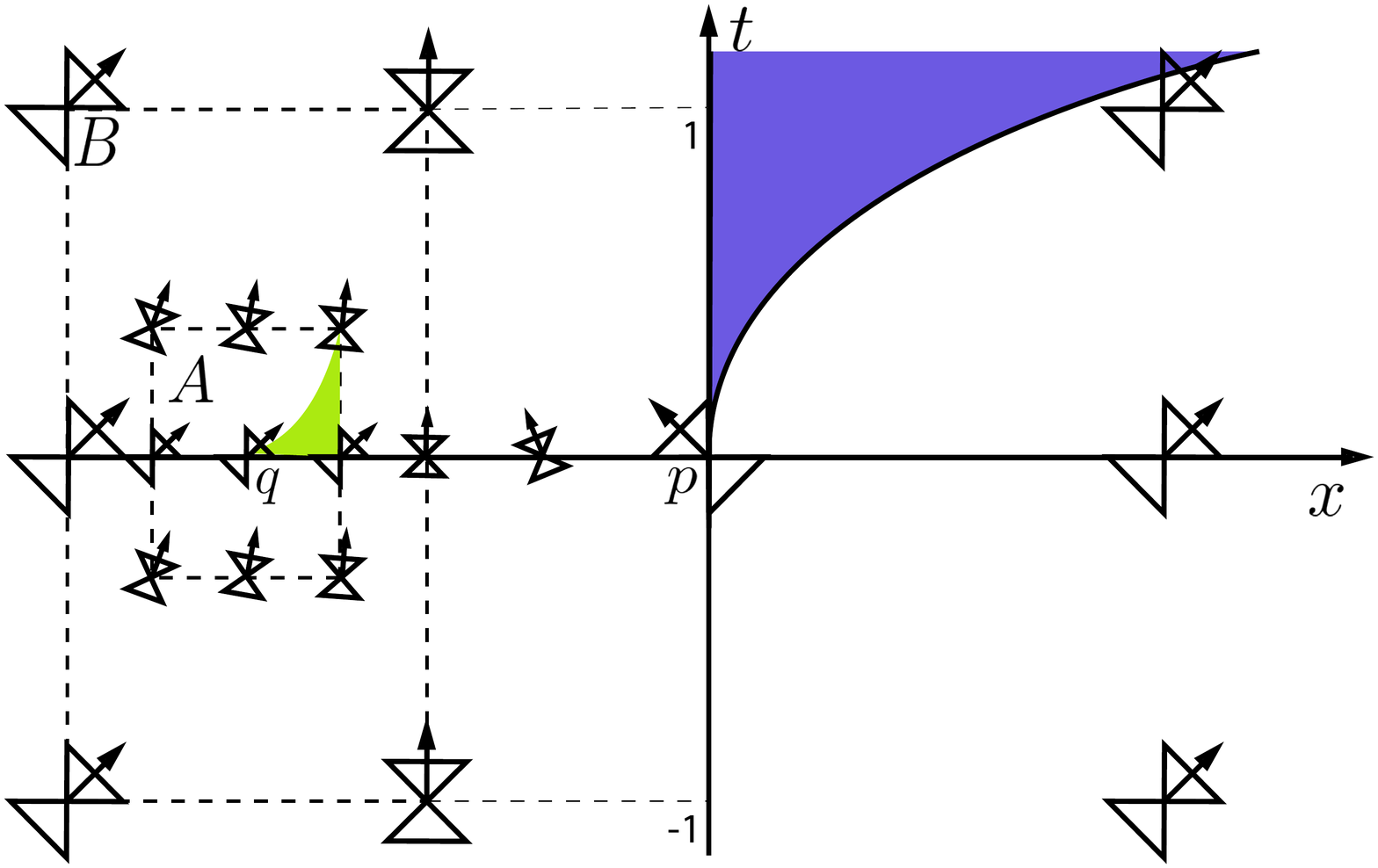}
	\end{center}
	\caption{Parts of the internal and external bubbles of $p$ and $q$, respectively, in Example~\ref{number3}.}
\label{fig:3}
\end{figure}

Let $q = (0, x_0)$ be a point in the interior of the set $A$. We denote by $\gamma_0$ the right-moving null generator from $q$ that immediately leaves the $x$-axis at $q$, i.e. $t(\gamma_0(s)) > 0$ for all $s > 0$. However, right-moving null generators from $q$ can travel along the $x$-axis for a finite time, and then leave the $x$-axis. We denote by $\gamma_1$ the right-moving null generator from $q$ that leaves the $x$-axis at the largest value of $x$, i.e. $\gamma_1$ leaves the $x$-axis at the point $(0, x_1)$ where $x_0 < -\frac{2}{3} \le x_1 < -\frac{1}{2}$. The null generators $\gamma_0$, $\gamma_1$ later intersect the $t$-axis at the points $(t_0, 0)$, $(t_1, 0)$, respectively, where $t_1 < t_0$. The set $J^+(q) \setminus \overline{I^+(q)}$ contains the area between $\gamma_0$ and $\gamma_1$ up to the $t$-axis, and hence has non-empty interior. As in Example~\ref{number2}, for any $\overline{t} > t_0$, the point $(\overline{t}, 0) \in I^+(q)$ and $I^+(\overline{t}, 0)$ contains a subset of the set $x 
> 0$ with non-empty interior. Since $\check{I}^+(q)$ is contained in the set $x < 0$, it follows that $I^+(q) \setminus \check{I}^+(q)$ has non-empty interior as well. 

This example therefore answers~(\ref{Q5}) in the negative, and indeed shows that there exist points $q$ in the spacetime such that $\mathcal{B}^+_{\mathrm{ext}}(q, U)$ is non-empty for some cylindrical neighbourhood $U$ of $q$, and $\check{I}^+(q) \neq I^+(q)$. 
 
\begin{remark}
In all of the above examples, the exceptional curves can, in fact, be chosen to be $\mathcal{C}^1$ with non-timelike tangent vector at only one point. Heuristically, one might expect that the seemingly innocuous change of allowing $\mathcal{C}^1$ curves that have timelike tangent vector except at a finite number of points at which the tangent vector may be null should not affect chronological pasts and futures. However, our examples show that in fact all of the pathologies exhibited above would persist if this definition were adopted. 
\end{remark}

\section{Conclusions}
The following diagram collects the (non-)implications between the causal regularity
properties we have studied in the previous sections. Let $(M,g)$ be a continuous spacetime and recall that 
$I^\pm(p):=I^\pm_{\mathcal{AC}}(p)=I^\pm_{\mathcal{L}}(p)$. Then:

\begin{center}
	\begin{tikzcd}[column sep = tiny]
		& & & \fbox{$g$ Lipschitz} 
		\arrow[d,rightarrow,shift right=1ex]  & &   & \\
		& & & \fbox{$g$ causally plain} 
		\arrow[d,leftrightarrow,"\textcircled{\raisebox{0pt}{$2$}}"] 
		\arrow[negated,u,rightarrow,shift right=1ex,"\hspace*{3pt} \textcircled{\raisebox{0pt}{$1$}}"']  & &   & \\
		& & & \fbox{$\forall (p,U): \mathcal{B}^\pm_{\mathrm{int}}(p,U) = \emptyset = \mathcal{B}^\pm_{\mathrm{ext}}(p,U)$}  
		\arrow[dl,rightarrow,shift right=2ex,shorten <=4ex,"\textcircled{\raisebox{0pt}{$3$}}"'] 
		\arrow[dr,rightarrow,shift left=2.2ex,shorten <=4ex,"\textcircled{\raisebox{0pt}{$4$}}"] & & &\\[3em]
		& & \fbox{$\forall (p,U): \mathcal{B}^\pm_{\mathrm{int}}(p,U) = \emptyset $ } \hphantom{x} 
		\arrow[negated,ur,rightarrow,shift right=0.1ex,shorten <=4ex] 
		\arrow[d,leftrightarrow,,"\textcircled{\raisebox{0pt}{$5$}}"]
		& & \fbox{$\forall (p,U): \mathcal{B}^\pm_{\mathrm{ext}}(p,U) = \emptyset $ } 
		\arrow[negated,ul,rightarrow,shift right=0.2ex,shorten <=4ex] 
		\arrow[d,leftrightarrow,"\textcircled{\raisebox{0pt}{$6$}}"] & &\\
		& &\fbox{ $\forall p\in M:\ I^\pm(p) = \check I^\pm(p) $} \hphantom{x} 
		\arrow[d,rightarrow,shift right=1ex] 
		\arrow[negated,rr,rightarrow,shift left=0.9ex,"\textcircled{\raisebox{0pt}{$7$}}"] & 
		& \fbox{ \vphantom{$I^\pm$} Push-up holds \hphantom{x}} 
		\arrow[negated,ll,rightarrow,shift left=0.9ex]
		\arrow[negated,lld,rightarrow,shift left=0.9ex,"\hspace*{5em}\textcircled{\raisebox{0pt}{$9$}}"'] & & \\
		& &\fbox{ $\forall p\in M:\ I^\pm(p)$ is open} 
		\arrow[negated,u,rightarrow,shift right=1ex,"\hspace*{3pt}\textcircled{\raisebox{0pt}{$8$}}"'] \hphantom{x} 
		\arrow[negated,rru,rightarrow,shift right=2.8ex] & & & & &\\
	\end{tikzcd}
\end{center}

Here, \textcircled{\raisebox{0.4pt}{\tiny  $1$}} was shown in~\cite[Cor.\ 1.17]{CG:12}. Irreversibility of this 
arrow can be seen by considering a metric that is conformally equivalent to the Minkowski metric via a non-Lipschitz factor. 
Of more interest is the fact that warped products with one-dimensional base (i.e., metrics of the form $-dt^2+f^2 h$ with $h$ any Riemannian metric and $f:(a,b) \to (0,\infty)$ continuous) are causally 
plain for continuous, not necessarily Lipschitz, warping functions $f$, cf.\ \cite{AGKS:19}.
Equivalence \textcircled{\raisebox{0.4pt}{\tiny  $2$}} is Corollary~\ref{cor:2.15}. Irreversibility in
\textcircled{\raisebox{0.4pt}{\tiny  $3$}}, \textcircled{\raisebox{0.4pt}{\tiny  $4$}} 
and \textcircled{\raisebox{0.4pt}{\tiny  $8$}}, as well as the claims in \textcircled{\raisebox{0.4pt}{\tiny  $9$}} follow from~\cite[Ex.\ 1.12]{CG:12} and Example~\ref{number2}, respectively. Equivalence \textcircled{\raisebox{0.4pt}{\tiny  $5$}} 
holds by Theorem~\ref{thm:all_i_equal_for_causally_plain}, while \textcircled{\raisebox{0.4pt}{\tiny  $6$}} is a consequence of 
Theorem~\ref{thm:push-up_ext_bubble}.
The claims in \textcircled{\raisebox{0.4pt}{\tiny  $7$}} follow from~\cite[Ex.\ 1.12]{CG:12}
and Example~\ref{number1}. Finally, Example~\ref{number2} entails irreversibility of \textcircled{\raisebox{0.4pt}{\tiny  $8$}}.

The watershed in this chain of implications is the case of causally plain spacetimes. Indeed, for such metrics all
versions of chronological futures and pasts coincide. 
Moreover, there are no bubble sets, chronological futures and pasts are open and the standard push-up
properties hold. Indeed, it was demonstrated in~\cite[Sec.\ 5.1]{KS:18} that strongly causal and 
causally plain spacetimes form strongly localisable Lorentzian length spaces, implying that their
causality theory is optimal in the sense of synthetic low regularity Lorentzian geometry.

On the other hand, the examples in Section~\ref{sec:counter_examples} demonstrate that continuous but non-Lipschitz spacetimes
display a number of unexpected new causal phenomena that are entirely absent from the causally plain setting.
The most drastic of these is the occurrence of non-open chronological futures/pasts. It is no exaggeration to
state that openness of $I^\pm(p)$ is ubiquitous throughout standard causality theory. For example, Theorem~\ref{th:iplusachronal} (iv) is a property that features prominently in proofs of the singularity theorems of
General Relativity: for the most general version of the causal part of these theorems we refer to~\cite[Sec.\ 2.15]{Min:19} (cf., in particular, the proof of Penrose's singularity theorem~\cite[Thm.\ 65]{Min:19}
that relies on this property in the form of~\cite[Thm.\ 20]{Min:19}). That this condition is (for chronological spacetimes) in fact equivalent to openness of chronological futures/pasts (Theorem~\ref{th:iplusachronal} (i)) 
is a strong indication that it is preferable to adopt a different definition for these sets in regularities below Lipschitz.
Indeed, the arguments put forward in this paper suggest that the optimal strategy for obtaining a
satisfactory causality theory is the following:
\begin{itemize}[itemsep=2mm]
\item In order to derive the maximal benefit from limit curve theorems, set $J^\pm(p):=J^\pm_{\mathcal{AC}}(p) = 
J^\pm_{\mathcal{L}}(p)$ (cf.\ Lemma~\ref{lem:lip_is_ac}).
\item To guarantee openness of chronological futures and pasts and to avoid interior bubbling, 
set $I^\pm(p):=I^\pm_{\copw}(p)$. Then in fact $I^\pm(p)=\check I^\pm(p) = I^\pm_{\cinfty}(p)$
by Lemma~\ref{lem:icheck_is_ic1}.
\end{itemize} 
The first point here is valid across all regularities of the metric. For the second one,
we have seen in Theorem~\ref{thm:all_i_equal_for_causally_plain} that for spacetimes without internal bubbling (in 
particular: for Lipschitz spacetimes) the choice of class of curves 
makes no difference, while the examples in Section~\ref{sec:counter_examples} clearly 
underline the advantage of adhering to this convention for non-Lipschitz metrics. 

While this strategy returns us to the ``unhandy situation in which timelike and causal paths have completely different
properties''~\cite[p.\ 14]{Chr:11}, it appears to us to be the most fitting approach. Moreover, note that, in fact, several fundamental
works on low regularity Lorentzian geometry have adopted these conventions, in particular~\cite{Min:15,Min:19, BS:18} and, in the timelike case,~\cite{GL:17,Sbi:18}. 
On the other hand, in synthetic approaches to low regularity Lorentzian geometry (such as~\cite{KS:18,GKS:18}), where
this strategy is not an option (due to the absence of a differentiable structure), 
phenomena such as the ones laid out in Section~\ref{sec:counter_examples} have to be taken into consideration.

\subsection*{Acknowledgement} We are indebted to two anonymous referees for several remarks that have led to considerable
improvements of the manuscript.

\end{document}